\documentclass[12pt]{amsart}

\usepackage{amsthm}
\usepackage{amssymb}
\usepackage{amsmath}
\numberwithin{equation}{section}
\newcommand{\bea}{\begin{eqnarray}}
\newcommand{\eea}{\end{eqnarray}}
\newcommand{\be}{\begin{eqnarray*}}
\newcommand{\ee}{\end{eqnarray*}}
\newtheorem{theorem}{Theorem}[section]
\newtheorem{lemma}{Lemma}[section]
\newtheorem{corollary}{Corollary}[section]

\newtheorem{example}{Example}[section]
\newtheorem{algorithm}{Algorithm}[section]

\begin{document}
\title[Dynamical Systems over Rings]{ Linear Dynamical Systems over \\ Finite Rings}
\author[Guangwu Xu]{Guangwu Xu}
\address{Department of Electrical Engineering and Computer Science, University of Wisconsin, Milwaukee, WI 53201, USA} \email{gxu4uwm@uwm.edu}
\author[Yi Ming Zou]{Yi Ming Zou}
\address{Department of Mathematical Sciences, University of Wisconsin, Milwaukee, WI 53201, USA} \email{ymzou@uwm.edu}
\maketitle
\begin{abstract}
The problem of linking the structure of a finite linear dynamical system with its dynamics is well understood when the phase space is a vector space over a finite field. The cycle structure of such a system can be described by the elementary divisors of the linear function, and the problem of determining whether the system is a fixed point system can be answered by computing and factoring the system's characteristic polynomial and minimal polynomial. It has become clear recently that the study of finite linear dynamical systems must be extended to embrace finite rings. The difficulty of dealing with an arbitrary finite commutative ring is that it lacks of unique factorization. In this paper, an efficient algorithm is provided for analyzing the cycle structure of a linear dynamical system over a finite commutative ring. In particular, for a given commutative ring $R$ such that $|R|=q$, where $q$ is a positive integer, the algorithm determines whether a given linear system over $R^n$ is a fixed point system or not in time $O(n^3\log(n\log(q)))$.
\end{abstract}
\section{Introduction}
A finite dynamical system is a function $f: X\longrightarrow X$, where $X$ is a finite set. The dynamics of the system is obtained by iterating the function $f$. Such dynamical systems have a variety of applications, such as in engineering, computer science, and computational biology [1,3,4]. 
\par It is a well-known fact in finite field theory that a function $f:\mathbb{F}_{q}^{n}\longrightarrow \mathbb{F}_{q}$, where $\mathbb{F}_q$ is a finite field of $q$ elements, can be represented by a polynomial function. Thus any function $f:\mathbb{F}_{q}^{n}\longrightarrow \mathbb{F}_{q}^{n}$ can be represented by $f=(f_1,\ldots, f_n)$, where $f_i\in \mathbb{F}_q{[x_1,\ldots, x_n]}$. When $f$ is a linear system, the dynamics of $f$ can be described using its characteristic polynomial and minimal polynomial, and the computation can be done in polynomial time [1,6]. For general polynomial systems, there have been only limited successes in determining the dynamics of such systems, except for monomial dynamical systems, where all the coordinate functions $f_i$ are monomials. 
\par
In [3], monomial dynamical systems over $\mathbb{Z}_2$, i.e. Boolean monomial systems, were studied. In [4], the problem of determining whether a monomial dynamical system over a finite field $\mathbb{F}_q$ is a fixed point system was reduced to the same question of an associated Boolean monomial system and a linear system over a ring of the form $\mathbb{Z}/(q-1)$. In [1], the study of fixed point systems was further developed. In particular, linear systems were defined for modules over a ring, and a necessary and sufficient condition for a linear system to be a fixed point system was derived using Fitting's lemma. 
\par
Though the result in [1] does not lead to an efficient algorithm for determining whether a linear system over a general finite commutative ring is a fixed point system, the computational problem, which is ultimately needed in applications, was discussed in some detail in the special case where the ring is a finite field, and a computational method via the factorization of the characteristic polynomial and the minimal polynomial of the linear function was described. As pointed out in [4], the approach via characteristic polynomial and minimal polynomial for a linear dynamical system over a finite commutative ring faces considerable difficulties due to the lack of unique factorization (see also the comment after Example 4 in [1]). The following example illustrates this point. 

\begin{example} Let $f:\mathbb{Z}_{8}^{2}\longrightarrow \mathbb{Z}_{8}^{2}$ be defined by the $2\times 2$ matrix 
\[ 
A=\left(\begin{array}{cc} 2 & 6\\ 1 & 0\end{array}\right). 
\]
Then $A^k\ne 0$ for $1\le k <6$ and $A^6=0$. Thus $f$ is a fixed point system with the only fixed point $0$. The characteristic polynomial of $A$ is $ch_{A}(\lambda)=\lambda^2+6\lambda+2$, which has no root in $\mathbb{Z}_8$, though $A$ clearly has an eigenvector $(0,4)^T$ corresponds to the eigenvalue $0$ (the eigenvalues and eigenvectors of a matrix over a commutative ring are defined as usual, see [2]). Note that 
\[
\lambda^6=(\lambda^2+6\lambda+2)(\lambda^4+2\lambda^3+2\lambda^2+4)\;\mbox{$\pmod{8}$}.
\]
\end{example}
\par
In this paper, we consider a different approach. Our approach is based on the fact that there are efficient algorithms for the computation of the powers of a matrix: the multiplication of two $n\times n$ matrices takes at most $n^3$ operations (the state of the art algorithms use close to $n^2$ operations). If $A$ is an $n\times n$ matrix, then to compute $A^m$, where $m$ is a positive integer, it will take about $n^3\log_2m$ operations. Therefore, one can just work with the matrix of a linear dynamical system directly to avoid the difficulties of dealing with the factoring problems over an arbitrary commutative ring.
\par
In order for this approach to work, one must have a reasonable upper bound on the exponents of the powers of the matrix, that is, a reasonable upper bound on the number of iterations, that one must compute in order to determine the dynamics of a given system. 
\par
Our first observation is, although Fitting's lemma tells us that a linear system will be stabilized after a certain number of iterations (see [1]), the lemma itself is a fairly general statement: it applies to any group $G$ that satisfies both ACC and DCC conditions on normal subgroups and any normal endomorphism $f$ of $G$ (see [7, p. 84]). While for the systems that we are interested in, the groups involved are finite abelian groups, and therefore, we should be able to derive more precise information on how many iterations it will need in order for a given system to reach a certain type of stabilization status.
\par
Our second observation is, the upper bound on the number of iterations also depends on the size of space. This can be seen from Example 1.1, where it takes $6$ iterations for the system to be stabilized. This can also be seen by just considering the simplest type of linear systems on $\mathbb{Z}_{q}$, where $q$ is a positive integer, namely the ones defined by a scalar multiplication. For such a system, the matrix size is 1, but the dynamics of the system depend on $q$. If the system is defined by the multiplication of an element $1<a<q$, then one either needs to know the prime factorizations of $a$ and $q$ or needs to compute the powers of $a$ to derive the dynamics of the system. Therefore, certain assumption on the size of $q$ must be made. Here we assume that the size of $q$ is comparable to the size of any integer that we maybe able to factor in the foreseeable future. We believe it is reasonable to make this assumption. Under this assumption, the numbers $\log_2q$ and $\log_2(\log_2q)$ are relatively small: the RSA keys are typically $1024-2048$ bits long and $\log_2(\log_22^{2048})=11$.  
\par
This paper is organized as follows. In Section 2, we develop the basic theory that lays the foundation for an efficient algorithm. In section 3, we describe an algorithm for determining whether a linear dynamical system over a finite ring is a fixed point system or not and give two examples of linear fixed point systems over finite rings which are not fields. In Section 4, we conclude with some discussions and an example.
\par\medskip
\section{Main results}
\par
Let $R$ be a finite commutative ring with $q>1$ elements. Let the prime factorization of $q$ be 
\be
q=\prod_{i=1}^{t}{p_{i}^{t_i}}.
\ee
\par
We shall view the elements of $R^n$, where $n$ is positive integer, as column vectors, and denote by $e_i$, $1\le i\le n$, the canonical basis (if $R$ has $1$). For a function $f$ from a set to the same set, we write
\be
f^m=\underbrace{f\circ f\circ\cdots\circ f}_{\mbox{$m$ copies}}
\ee
if $m$ is a positive integer. If $m$ is a positive number, not necessary an integer, then by writing $f^m$ we mean $f^{\lceil m\rceil}$, where $\lceil m\rceil$ is the smallest integer greater than or equal to $m$. Our first theorem upper bounds the number of iterations needed for a linear system to reach a certain stable status.
\par
\begin{theorem} Let $n$ be a positive integer, and let $f:R^n\longrightarrow R^n$ be a linear function. Then for any nonnegative integer $k$, we have 
\be
f^{n\log_{2}(q)+k}(R^n)=f^{n\log_{2}(q)}(R^n).
\ee
If $R$ is a field, then the factor $\log_{2}(q)$ is not needed, that is
\be
f^{n+k}(R^n)=f^{n}(R^n).
\ee
\end{theorem}
\begin{proof} We first consider the general case when $R$ is a commutative ring. View $R^n$ as an $f$-module, set $M_0=R^n$, and consider a sequence of $f$-submodules of $M_0$ defined by
\bea
M_0\supseteq M_1=f(M_0)\supseteq\cdots\supseteq M_r=f^r(M_0)\supseteq\cdots.
\eea
Since each $M_r$ ($r\ge 0$) is a finite abelian group and 
\be
|M_0|=q^n=\prod_{i=1}^{t}{p_{i}^{nt_i}},
\ee
by Lagrange's theorem, we have
\be
|M_r|=\prod_{i=1}^{t}{p_{i}^{r_i}},
\ee
where $0\le r_i\le nt_i$. Thus, if $M_r\ne M_{r+1}$, then 
\be
|M_{r+1}|\le |M_r|/p_i
\ee
for some $1\le i\le t$. Therefore either there is an
\bea
r<\sum_{i=1}^{t}{nt_i}=n\sum_{i=1}^{t}{t_i}:=s,
\eea
such that $M_r=M_{r+1}$, or we must have $|M_s|=1$. In any case, $f(M_s)=M_s$. Since $s\le n\log_{2}(q)$, the first statement follows.
\par
If $R$ is a field, then the modules $M_i$ are vector spaces over $R$. So if $M_i\supsetneq M_{i+1}$, then $\dim M_{i+1}\le\dim M_{i}-1$. Since $\dim M_0=n$, the desired result follows. 
\end{proof}
\par
Next, we give a general lemma about fixed point systems on a finite set. We remark that one can almost read out the proof of the lemma from the proof of Theorem 2 in [1]. Here we give a proof which sheds some light from a different view.
\begin{lemma} If $X$ is a finite set and $f:X\longrightarrow X$ is a function such that $f(X)=X$, then $f$ is a fixed point system if and only if $f$ is the identity function.
\end{lemma}
\begin{proof} Since $X$ is a finite set, $f(X)=X$ implies that $f$ is also injective. Thus $f$ is a permutation of the set $X$. Writing $f$ as a disjoint product of cycles, we see immediately that $f$ is a fixed point system if and only if all the cycles have length one, that is, $f$ is the identity function. 
\end{proof}
\par
Recall that an element $u$ in a commutative ring $R$ with $1$ is called a unit if it is invertible. The following is an immediate consequence of Lemma 2.1.
\begin{corollary}
Let $R$ be a finite commutative ring with $1$. Let $A:R^n\longrightarrow R^n$, where $A$ is an $n\times n$ matrix over $R$, be a linear dynamical system. If $A\ne I$ and $\det A$ is a unit in $R$, then $A$ is not a fixed point system.
\end{corollary} 
Now we give a criterion for a linear dynamical system over a finite ring to be a fixed point system.
\begin{theorem} Let $R$ be a finite commutative ring with $q$ elements, let $n$ be a positive integer, let $f:R^n\longrightarrow R^n$ be a linear system, and let $A$ be the matrix of $f$ with respect to the canonical basis (if $R$ has $1$). Then $f$ is a fixed point system if and only if $f^{n\log_{2}(q)+1}=f^{n\log_{2}(q)}$, or equivalently $A^{n\log_{2}(q)+1}=A^{n\log_{2}(q)}$. If $R$ is a field, then the condition simplifies to $f^{n+1}=f^{n}$ or $A^{n+1}=A^{n}$.
\end{theorem}
\begin{proof} By Theorem 2.1, 
\be
f(f^{n\log_{2}(q)}(R^n))=f^{n\log_{2}(q)}(R^n).
\ee
Thus, by Lemma 2.1, $f$ is a fixed point system if and only if 
\be
f|_{f^{n\log_{2}(q)}(R^n)}=id|_{f^{n\log_{2}(q)}(R^n)},
\ee
which is equivalent to 
\be
f(f^{n\log_{2}(q)}(x))=f^{n\log_{2}(q)}(x),\quad \forall x\in R^n.
\ee 
That is $f^{n\log_{2}(q)+1}=f^{n\log_{2}(q)}$.
\end{proof}
\par
Theorem 2.2 provides an efficient algorithm to determine whether a linear dynamical system over a finite commutative ring is a fixed point system, which will be discussed in the next section. The results in this section also reduce the study of a general linear dynamical system over a finite commutative ring to an invertible non-fixed point system. 
\medskip
\section{Algorithms and Examples}
In this section, we first describe an algorithm based on Theorem 2.2 for determining
whether a linear system $A:\mathbb{Z}_{q}^{n}\longrightarrow\mathbb{Z}_{q}^{n}$ is a fixed point system or not, where $q>1$ is an integer and $A$ is taken to be the form of an $n\times n$ matrix. We choose $\mathbb{Z}_{q}$ as the base ring for the simplicity of the statements, the same algorithm works for any ring of the type 
\be
\mathbb{Z}_{q_1}\times \mathbb{Z}_{q_2}\times\cdots\times\mathbb{Z}_{q_k},
\ee
as well as for any finite commutative ring with $1$ as long as the operations of the ring are implemented.
\par
The algorithm is called an {\sl LFPS (Linear Fixed Point System) test}.
\begin{algorithm} {\bf LFPS test}.\\
{\bf Input:} Two positive integers $n$ and $q>1$, an $n\times n$ matrix $A$ over $\mathbb{Z}_{q}$, and $b_{t-1}2^{t-1}+b_{t-2}2^{t-2}\cdots +b_12+b_0$, the binary representation of $\lceil n\log_2 q\rceil$.\\
{\bf Output:} {\bf true} or {\bf false}. \\
\begin{enumerate}
\item $X\gets I$
\item {\bf for} $i$ from $t-1$ down to $0$ {\bf do}\\
\mbox{}\hspace{5mm} $X\gets XX$\\
\mbox{}\hspace{5mm} {\bf if} $b_i = 1$ {\bf then}\\
\mbox{}\hspace{10mm} $X\gets AX$\\
\item {\bf if} $X=XA$ {\bf then}\\
\mbox{}\hspace{5mm} {\bf return true}\\
{\bf else}\\
\mbox{}\hspace{5mm} {\bf return false}\\
\end{enumerate}
\end{algorithm}
Let us explain this algorithm in more detail.  In step (1) the
(matrix) variable $X$ is initialized by the identity matrix $I$. The
main computation of $A^{\lceil n\log_2 q\rceil}$ is performed in
step (2) using the ``square and multiply'' method.
Since $b_{t-1} = 1$ (the leading bit of $\lceil n\log_2
q\rceil$), at the beginning (i.e., $i=t-1$), $X$ first becomes
$II=I$, then becomes $X = AI=A$. After this, for each $i$ with
$t-2\ge i\ge 0$, the value in $X$ becomes the square of the value
previously stored in $X$. If $b_i=1$, then the value of $X$ is
further updated to be the product of $A$ and the previous
value. At the end of step (2), the value in $X$ is $A^{\lceil
n\log_2 q\rceil}$. For example, if $A$ is a $6\times 6$ matrix over
$\mathbb{Z}_{3\cdot 7}$, then $\lceil 6\log_2 21\rceil = 27$ and by the
``square and multiply'' method:
\[
A^{27} = A^{1\cdot 2^4+1\cdot 2^3+0\cdot 2^2+1\cdot 2+1} =
\bigg(\bigg(\big((A)^2A\big)^2 \bigg)^2A\bigg)^2A.
\]
In step (3), the result of Theorem 2.2 is applied. Since the value
of $X$ is now $A^{\lceil n\log_2 q\rceil}$, the system is a
fixed point system if $X=XA$, and the program returns {\bf true};
otherwise, the system is not a fixed point system and the program
returns {\bf false}.
\par
\medskip
Suppose two matrices over  $\mathbb{Z}_{q}$ can be multiplied
with $O(n^{\omega})$ operations, by using Strassen's algorithm,
$\omega \le \log_27$. This number can be further reduced, see
[5]. The cost of running {\sl LFPS test} is
$O(n^{\omega}(\log_2n +\log_2\log_2 q))$. Under our assumption that
$\log_2\log_2 q$ is small, determine whether a linear system over a finite ring is a fixed point system or not can be done with $O(n^3)$ operations. If $R$ is a field, then the number of operations required is $O(n^{\omega}\log_2n)$.

As long as the problem of determining whether a linear system is a fixed point system is concerned, a comparison of the computational cost analysis given in [1] with the analysis given above shows, in addition to its simplicity, that our algorithm is at least as efficient as the approach via the characteristic polynomial and the minimal polynomial even for the case of finite fields. 

Next we give two examples of fixed point linear systems over
finite rings. The first example is over the ring $\mathbb{Z}_{2^4}$.
\begin{example}
The system $A:\mathbb{Z}_{2^4}^4\longrightarrow \mathbb{Z}_{2^4}^4$ defined by
\[
A=\begin{pmatrix}
  15 & 7 & 7 & 1\\
  0 & 7 & 11 & 7\\
  7 &7 &7 &11\\
  14&8 &15 &6
\end{pmatrix},
\]
is a fixed point system. This can be verified by using Algorithm 3.1 to compute $A^{4\log_2 2^4}=A^{16}$ ($4$ iterations) and verify that
$A^{16}=A^{17}$. The ``stabilized'' matrix is
\[
A^{16} =\begin{pmatrix}
  12 & 1 & 2 & 11\\
  0 & 4 & 8 & 12\\
  4 & 3 & 6 & 1\\
  12& 1 & 2 & 11
\end{pmatrix}.
\]
\end{example}

The second example describes a fixed point system over the ring
$\mathbb{Z}_{3^2\cdot 5}$.
\begin{example}
The system $A:\mathbb{Z}_{3^2\cdot 5}^4\longrightarrow\mathbb{Z}_{3^2\cdot 5}^4$ defined by
\[
A=\begin{pmatrix}
36 &23 &32 &9\\
27 &32 & 30& 25\\
32& 25& 13 &28\\
32 &8& 41& 40
\end{pmatrix}.
\]
is a fixed point system. The ``stabilized'' matrix is
\[
A^{4\lceil\log_2(3^2\cdot 5)\rceil}=A^{24}=\begin{pmatrix}
0 &9& 9& 27\\
10 &27 &12 &26\\
35 &18 &33 &19\\
5 &27 &42 &31
\end{pmatrix}
\]
\end{example}

We remark that the number $r$ such that $A^r=A^{r+1}$ can be smaller than our theoretical bound $n\log_2q$ in some cases. In the second example above, $r=6$, i.e., we have $A^6=A^7$. Our algorithm can be refined so it terminates before the iteration process reaches the theoretical bound if $r$ is small enough, say $r<\sqrt{n\log_2q}$. But we believe that the gain is not significant by doing so.
\section{Conclusions}
We have provided an efficient algorithm to determine whether a linear dynamical system over a finite commutative ring is a fixed point system. As an application, our result together with the results in [3] and [4] should settle the problem of determining whether a monomial dynamical system over a finite field is a fixed point system.
\par
When the system is not a fixed point system, a natural problem is finding the cycles of the system. If $R$ is a field, then under the assumption that the elementary divisors of a linear system and their orders (the order of a polynomial $g$ is the least positive integer $k$ such that $g$ divides $x^k-1$) can be computed efficiently, the cycles can be computed by a theorem due to Elspas (see [6]). Obviously, the implementation via such approach is quite involved, in particular the computation of the orders of the elementary divisors. The orders of the elementary divisors are the lengths of the cycles. If the lengths of the cycles can be found, then the cycles can be obtained. For example, suppose that $f$ is linear dynamical system over a finite commutative ring $R$ with $1$, and suppose that the lengths of its cycles, say 
\be
0=k_0<k_1<\ldots<k_m, 
\ee
are known. Then the cycles can be computed by solving the linear systems:
\be
(A^{k_i}-I)X=0,\quad 0\le i\le m.
\ee
\begin{example} Consider the system $A:\mathbb{Z}_{105}^4\longrightarrow\mathbb{Z}_{105}^4$ defined by
\[
A=\begin{pmatrix}
70 &27 &5 &26\\
35 &98 & 104& 99\\
81& 85& 78 &102\\
27 &97& 13& 69
\end{pmatrix}.
\]
Since $\det A=2 \pmod{105}$ is a unit in $\mathbb{Z}_{105}$, Corollary 2.1 implies that $A$ is not a fixed point system. Since $A^{24}=I$ and $A^k\ne I$ for $0<k<24$, we see that the cycles lengths are the factors of $24$. With some computation, one can find the cycles lengths, they are $1, 2, 24$. The only cycle of length 1 is $0$, there are $5512$ cycles of length 2, and $5064150$ cycles of length $24$.
\end{example}
However, the search for the cycle lengths seems to be exponential. 
\par\medskip
Computations of linear systems over finite commutative rings are basic, since one typically handles the other computational problems by reducing them to the ones about linear systems, and for systems over finite fields, the reduction can result in linear systems over commutative rings which are not necessary fields.  Developing efficient algorithms over commutative rings deserves further attention (see also [1]). 
\subsection*{Acknowledgment}
The first author gratefully acknowledges partial support from 
the National 973 Project of China (No. 2007CB807900).
 
\end{document}